\documentclass[11pt, reqno]{amsart}
\usepackage{euscript,amscd,amsgen,amsfonts,amssymb,latexsym,graphicx,psfrag}

\newcommand{\eqdef}{\stackrel{\scriptscriptstyle\rm def}{=}}

\newtheorem{theorem}{Theorem}

\newtheorem{proposition}{Proposition}
\newtheorem{corollary}{Corollary}
\newtheorem{definition}{Definition}

\newtheorem*{remarks}{Remarks}
\newtheorem{example}{Example}

\newcommand{\beha}{\begin{enumerate}}
\newcommand{\behe}{\end{enumerate}}
\renewcommand{\epsilon}{\varepsilon}

\newcommand{\bmin}{{ a_{w}}}
\newcommand{\bmax}{{b_{w}}}

\newcommand{\R}{{\rm Rot}}

\newcommand{\cM}{\EuScript{M}}

\newcommand{\bR}{{\mathbb R}}

\newcommand{\bZ}{{\mathbb Z}}
\newcommand{\bN}{{\mathbb N}}

\newcommand{\cA}{{\mathcal A}}
\newcommand{\cB}{{\mathcal B}}

\def\1{1\!\!1}

\def\and{\text{ and }}

                        \def\^{\tilde}

\def\STP{{\rm{(STP)}}}

\def\1{1\!\!1}

\def\rv{{\rm rv}}

\newtheorem*{thmA}{Theorem A}

\newtheorem*{thmB}{Theorem B}

\DeclareMathSymbol{\varnothing}{\mathord}{AMSb}{"3F}
\renewcommand{\emptyset}{\varnothing}
\title{Localized Pressure and equilibrium states}
\author{Tamara Kucherenko}\address{Department of Mathematics, The City College of New York, New York, NY, 10031, USA}\email{tkucherenko@ccny.cuny.edu}

\author{Christian Wolf}\address{Department of Mathematics, The City College of New York, New York, NY, 10031, USA}\email{cwolf@ccny.cuny.edu}

\thanks{This work was partially supported by a grant from the Simons Foundation (\#209846 to Christian Wolf).}

\begin{document}

\begin{abstract}
We introduce the notion of localized topological pressure for continuous maps on compact metric spaces. The localized pressure of a continuous potential $\varphi$ is computed by considering only those $(n,\epsilon)$-separated sets whose statistical sums with respect to an $m$-dimensional potential $\Phi$ are "close" to a given value $w\in \bR^m$.
We then establish for several classes of systems and potentials $\varphi$ and $\Phi$  a local version of the variational principle. We also construct examples showing that the assumptions in the localized variational principle are fairly sharp. Next, we study localized equilibrium states and show that even in the case of subshifts of finite type and H\"older continuous potentials, there are several new phenomena that do not occur
in the theory of classical equilibrium states. In particular, ergodic localized equilibrium states for H\"older continuous potentials are in general not unique.
\end{abstract}
\subjclass[2000]{}
\maketitle

\section{Introduction}
\subsection{Motivation }
The thermodynamic formalism has been an important tool in the development of  the theory of dynamical systems. Originally, this subject was primeraly driven
by applications in  dimension theory  that  followed the pioneer works  carried out by Ruelle, Bowen and Manning and McCluskey \cite{B1, Ru2,MM}. These works inspired numerous studies and generalizations with applications far  beyond the sole focus on dimension.  For example, pressure can be applied to obtain information about Lyapunov exponents, dimension, multifractal spectra, or natural invariant measures.
 We refer to  \cite{BG,P,PU, Ru} for details and further references.

The main object in the thermodynamic formalism is the topological pressure, a certain functional
defined on the space of observables that encodes several important quantities of the underlying dynamical system.  The relation between the topological pressure and invariant measures is established by the variational principle. Namely, if $f:X\to X$ is a continuous map on a compact metric space and  $\varphi:X\to\bR$ is a continuous potential, then the topological pressure  $P_{\rm top}(\varphi)$ is given by the supremum of the free energy of the invariant probability measures (see \eqref{eqvarpri} for the precise statement). This result is powerful in part because it connects in a natural but unexpected way topological and statistical dynamics. Invariant probabilities maximizing free energy are called equilibrium states. The study of equilibrium states (existence, uniqueness and properties) has a long history and the results are widely spread in the literature, yet a complete understanding is still lacking today. We refer to \cite{Bo1,Bo2,HRu,Ke} for references and details.

Our focus in this paper is somewhat different. We introduce a  localized version of the topological pressure where the localization results from using  only those orbits in the computation of the pressure whose statistical averages with respect to a given $m$-dimensional potential $\Phi$ are close to a vector $w\in \bR^m$. We  then establish a version of the localized variational principle for
a wide variety of systems and potentials. We also show that the assumptions in our localized variational principle are fairly sharp. Finally, we develop the theory of localized equilibrium states and derive non-uniqueness results for these equilibrium states. Our results significantly distinguish localized equilibrium states from  the theory of classical equilibrium states.

The results in this paper are related and can be considered in some sense extensions of results in the higher dimensional multifractal analysis developed by Barreira, Saussol, Schmeling, Takens, Verbitskiy, and others (see for example \cite{BSS,BS,
TV}). For  localizations using  restrictions of  the pressure to non-compact subsets we refer to \cite{C,PP,T} and the references therein.
We will now describe our results in more detail.

\subsection{ Basic definitions and statement of the results. }
Let $f:X\to X$ be a continuous map on a compact metric space $(X,d)$. We consider continuous potentials $\varphi:X\to\bR$ and $\Phi=(\phi_1,\cdots,\phi_m):X\to \bR^m$. We think of $\varphi$ as our target potential for computing the localized topological pressure and of $\Phi$ as the potential providing the localization. For $n\in \bN$ and $\epsilon>0$, we say that  $F\subset X$ is $(n,\epsilon)$-separated if for all $x,y\in F$ with $x\not= y$ we have $d_n(x,y)\eqdef \max_{k=0,\cdots,n-1} d(f^k(x),f^k(y))\geq \epsilon$. Note that $d_n$ is a metric (called Bowen metric) that induces the same topology on $X$ as $d$.
For $x\in X$ and $n\in \bN$, we denote by $\frac{1}{n}S_n\Phi(x)$
the $m$-dimensional Birkhoff average at $x$ of length $n$ with respect to $\Phi$, where
\begin{equation}\label{defSnm}
S_n\Phi(x)=\left(S_n\phi_1(x),\ldots,S_n\phi_m(x)\right)
\end{equation}
and $S_n\phi_i(x)=\sum_{k=0}^{n-1} \phi_i(f^k(x))$.
Given $w\in \bR^m$ and $r>0$ we say a set $F\subset X$ is a $(n,\epsilon,w,r)$-set if $F$ is $(n,\epsilon)$-separated set and  for all $x\in F$ the Birkhoff average $\frac{1}{n}S_n\Phi(x)$ is contained in
the Euclidean ball $D(w,r)$ with center $w$ and radius $r$.
We define the \emph{localized topological pressure} of the potential $\varphi$ (with respect to $\Phi$ and $w$) by
\begin{equation}\label{loctopprint}
P_{\rm top}(\varphi,\Phi,w)=
\lim_{r\to 0}\lim_{\varepsilon \to 0}
            \limsup_{n\to \infty}
            \frac{1}{n} \log N_{\varphi}(n,\varepsilon, w, r),
\end{equation}
where
\begin{equation} N_{\varphi}(n,\varepsilon, w, r)=\\ \sup \left\{\sum_{x\in F}e^{S_n\varphi(x)}:\, F\text{ is } (n,\varepsilon,w,r)\text{-set }\right\}.
\end{equation}
This definition is analogous to that of the classical topological pressure with the exception that we here only consider orbits with Birkhoff averages  close to $w$.  Moreover, when we omit the limit $r\to 0$ in \eqref{loctopprint} and choose $r$ large enough that the range of $\Phi$ is contained in $D(w,r)$, then we obtain the classical topological pressure of $\varphi$.

Note that the definition of $P_{\rm top}(\varphi,\Phi,w)$ is only meaningful if $D(w,r)$ contains statistical averages with respect to $\Phi$ for infinitely many $n$ and arbitrarily small $r$.
We call the corresponding set of points $w$  the {\it pointwise rotation set} of $\Phi$ and denote it by $\R_{Pt}(\Phi)$, that is
 \begin{equation}\label{defRf}
\R_{Pt}(\Phi)=
\left\{w\in \bR^m: \forall r>0\ \forall N\ \exists n\geq N\ \exists\ x\in X: \  \frac{1}{n}S_n\Phi(x)\in D(w,r)\right\}
\end{equation}

Next, we discuss a measure-theoretic approach to rotation sets and localized pressure. We denote by $\cM$ the set of all Borel $f$-invariant probability measures on $X$ endowed with the weak$^\ast$ topology. Following \cite{Je}, we define the {\it generalized rotation set} of  $\Phi$   by
\begin{equation} \label{defrotset}
 \R(\Phi)= \left\{\rv(\mu): \mu\in\cM\right\},
\end{equation}
where $\rv(\mu)=\left(\int \phi_1\ d\mu,\ldots,\int \phi_m\ d\mu\right)$
 denotes the rotation vector of the measure $\mu$. We call $\cM_\Phi(w)=\{\mu\in \cM: \rv(\mu)=w\}$ the rotation class of $w$.
In \cite{KW} we study the relationship between the pointwise rotation set and generalized rotation set of $\Phi$. In particular, we show that   $\R_{Pt}(\Phi)\subset \R(\Phi)$ with strict inclusion in certain cases.  We also provide criteria for the equality of the two rotation sets.
We refer to the overview article \cite{Mi1} and to \cite{Je,KW,Z} for
further details about rotation sets.
For $w\in\R(\Phi)$, we define the \emph{localized measure-theoretic pressure} of the potential $\varphi$ (with respect to $\Phi$ and $w$) by

\begin{equation}\label{locmespress}
P_{\rm m}(\varphi,\Phi,w)=\sup\left\{h_\mu(f)+\int_X\varphi\,d\mu:\,\mu \in \cM_\Phi(w)\right\}.
\end{equation}
In case we take the supremum in \eqref{locmespress} over all invariant measures we obtain the classical measure-theoretic pressure.
The classical variational principle (without localization) states that the topological and   the measure-theoretic versions of the pressure coincide. However, it turns out that in the case of localized pressure, the measure-theoretic and topological pressures may differ and strict inequalities can occur in both directions. This follows from the Examples 1 and 2  given in Section 3. On the other hand,  the following
result (see Theorem \ref{theovar} in the text) gives a fairly complete description of the assumptions needed to still have a variational principle.

\begin{thmA} Let $f:X\to X$ be a continuous map on a compact metric space $X$ that is a Besicovitch space.  Let $\varphi:X\to\bR$ and $\Phi:X\to\bR^m$ be continuous
and let $w\in \R(\Phi)$ be such that the map $v\mapsto P_m(\varphi,\Phi, v)$ is continuous at $w$ and $P_{\rm m}(\varphi,\Phi, w)$ is approximated by ergodic measures. Then $P_{\rm top}(\varphi,\Phi, w)=P_{\rm m}(\varphi,\Phi, w)$.
\end{thmA}
The assumption that $P_{\rm m}(\varphi,\Phi, w)$ is approximated by ergodic measures (see Section 3 for the precise definition) cannot be dropped in Theorem A. Indeed,  Example 1 does not satisfy this assumption and $P_{\rm top}(\varphi,\Phi,w)<P_{\rm m}(\varphi,\Phi,w)$ holds. On the other hand, without the assumption that $v\mapsto P_m(\varphi,\Phi, v)$ is continuous at $w$, Theorem A is in general not true, which is a consequence of Example 2.
We recall that the continuity of $v\mapsto P_m(\varphi,\Phi, v)$ holds for all $w\in \R(\Phi)$ if the entropy map $\mu\mapsto h_\mu(f)$ is upper semi-continuous. In particular, this is true if $f$ is expansive \cite{Wal:81}, a $C^\infty$ map on a compact smooth Riemannian manifold \cite{N} or satisfies the entropy-expansiveness (as for example certain partial hyperbolic systems \cite{FDPV}). Recently, there has been significant progress in finding milder conditions that imply the upper-semicontinuity of the entropy function (see for example \cite{CT}).

We note that Theorem A holds for a wide variety of systems and potentials. In particular, Theorem A holds for systems with strong thermodynamic properties (STP) (see Section 3).

Next, we present our results about localized equilibrium states. Fix $w\in \R(\Phi)$.
We say  $\mu\in \cM_\Phi(w)$ is a localized equilibrium state of $\varphi\in C(X,\bR)$ (with respect to $\Phi$ and $w$) if
\begin{equation}\label{defeqsta}
h_\mu(f)+ \int_X \varphi\ d\mu =
\sup_{\nu\in \cM_\Phi(w)} \left(h_\nu(f)+\int_X \varphi\,d\nu\right).
\end{equation}
This definition is analogous to that of  a classical equilibrium state with the  exception that we here only consider  invariant measures in $\cM_\Phi(w)$ rather than all invariant measures. Evidently, the upper semi-continuity of the entropy map guarantees the existence of at least one localized equilibrium state. Unlike in the case of classical equilibrium states, there does not need to exist an ergodic localized equilibrium state (see Example 3).
In Section 4 we introduce the class of systems with strong thermodynamic properties that include subshifts of finite type, hyperbolic systems and expansive homeomorphisms with specification. These systems exhibit the strongest possible properties for classical equilibrium states. In particular,  for each H\"older continuous potential $\varphi$, there exists a unique equilibrium state $\mu_\varphi$ (which is ergodic) and $\mu_\varphi$ has the Gibbs property. We show that  this result does not carry over to localized equilibrium states. In Example 4 we consider a shift map
 and construct a Lipschitz continuous potential $\Phi$ exhibiting exactly two ergodic localized equilibrium states, none of which is Gibbs. We call the corresponding rotation set the "fish" due its shape. We study this example in great detail and derive properties that can be used to construct further counter examples.
Indeed, we are able to prove that the boundary of the fish is an infinite polygon and compute an exact formula for the corresponding vertices.
By slightly modifying this example, we show that the cardinality of ergodic localized equilibrium states is in general not preserved under small perturbations of the potential.
All these examples are formulated for $\phi\equiv 0$ (i.e. the localized entropy) and $w\in \partial \R(\Phi)$. In  Theorem B (i) (see below), we show that these phenomena do not not occur if $w\in$ int $\R(\Phi)$.

This motivates the following definition: Let $\mu$ be a localized equilibrium state of $\varphi$ (with respect to $\Phi$ and $w$). We say $\mu$ is an
  {\it interior localized equilibrium state}  if  $(\int\varphi\ d\mu,w)\in {\rm ri}\ \R(\varphi,\Phi)$ (where ${\rm ri}$ denotes the relative interior of the set), otherwise we say $\mu$ is a {\it localized equilibrium state at the boundary}. Without loss of generality we can always assume that $\dim \R(\Phi)=m$ (i.e. $\R(\Phi)$ has non-empty interior $\bR^m$) because otherwise we could just consider a lower dimensional affine subspace.
The following result shows that interior equilibrium states still share many of the properties of classical equilibrium states.

\begin{thmB}
Suppose that $f:X\to X$ is a system with strong thermodynamic properties. Let $\varphi$ and $\Phi$ be H\"older continuous potentials, and let $w\in {\rm int }\  \R(\Phi)$.
Then
\begin{enumerate}
\item[(i)] If $\dim \R(\varphi,\Phi)=m$, then there exists a unique (ergodic) localized equilibrium state at $w$.
\item[(ii)]
 Suppose $\dim \R(\varphi,\Phi)=m+1$ and that
all localized equilibrium states of $\varphi$ are interior equilibrium states. Then the set of ergodic localized equilibrium states  is  non-empty and finite.
  \item[(iii)]
 Under each of the assumptions {\rm (i)} or {\rm (ii)}, every ergodic localized equilibrium state $\mu_\varphi$ is a classical equilibrium state of the potential $s\varphi+t\cdot \Phi$ for some $s\in \bR$ and $t\in \bR^m$.
\end{enumerate}
\end{thmB}

We note that part (i) of Theorem B holds in particular for $\varphi\equiv 0$ (and more generally if $\varphi$ is cohomologous to a constant). Therefore, the assumption $w\in {\rm int}\ \R(\Phi)$ implies the existence of an unique localized measure of maximal entropy. Another interesting feature of Theorem B is that in both cases, (i) and (ii) the ergodic localized equilibrium state is a classical equilibrium state. This implies that if $f$ is a subshift of finite type, a uniformly hyperbolic system or an expansive homeomorphism with specification, any ergodic localized equilibrium state is a Gibbs state.

The proof of Theorem B  relies heavily on methods from the thermodynamic formalism and, in particular, on the analyticity
of the topological pressure for H\"older continuous potentials. Moreover, we use results of Jenkinson \cite{Je} as  key ingredients.

This paper is organized as follows: In Section 2, we review some background material.
Section 3 is devoted to the proof of the localized variational principle (Theorem A) and  the construction of certain  examples showing that without  the assumptions of Theorem A, the localized variational principle fails. Finally, in Section 4 we discuss localized equilibrium states and discover fundamental differences between the theory of classical and localized equilibrium states. In particular, we prove Theorem B for systems with strong thermodynamic properties.

%---------------------------------------------------------------------------
\section{Preliminaries}
%---------------------------------------------------------------------------

In this paper we consider deterministic discrete-time dynamical systems given by a continuous map $f:X\to X$ on a compact metric space $(X,d)$. We are concerned with a continuous potential $\varphi:X\to \bR$ and an $m$-dimensional continuous potential $\Phi=(\phi_1,\ldots,\phi_m):X\to \bR^m$. Consider the set $\cM$ of all Borel $f$-invariant probability measures endowed with weak$^*$ topology and denote by $\cM_E\subset \cM$ the subset of ergodic measures. We recall the definition of the pointwise rotation set  $\R_{Pt}(\Phi)$ (see \eqref{defRf}) and the rotation set $\R(\Phi)$ (see \eqref{defrotset}). Similarly, the {\it ergodic rotation set} is defined by
\begin{equation}
\R_E(\Phi)= \left\{\rv(\mu): \mu\in\cM_E\right\}.
\end{equation}
Rotation sets originated from Poincar\'e's rotation numbers for circle homeomorphisms \cite{Po}. The relation between the three different rotation sets is studied in detail in \cite{KW}.  Both, $\R_{Pt}(\Phi)$ and $\R(\Phi)$ are compact and $\R(\Phi)$ is convex. We always have
\begin{equation}
\R_{E}(\Phi)\subset \R_{Pt}(\Phi)\subset \R(\Phi),
\end{equation}
 where both inclusions can be strict. The first inclusion follows from Birkhoff's Ergodic Theorem and the second is a consequence of the sequential compactness of $\cM$ (see \cite{KW} for details).

For completeness we now recall the notion of the classical topological pressure.   For $n\in \mathbb N$ and $\varepsilon >0$   let \begin{equation} N_{\varphi}(n,\varepsilon)=\sup \left\{\sum_{x\in F}e^{S_n\varphi(x)}\, :\, F\subset X \text{ is } (n,\varepsilon)\text{-separated}\right\}.\end{equation}
The \emph{topological pressure} with respect to the dynamical system $(X,f)$ is a mapping  
$ P_{\rm top}(f,\cdot)\colon C(X,\bR)\to \bR\cup\{\infty\}$  defined by
\begin{equation}\label{defdru}
  P_{\rm top}(\varphi) = \lim_{\varepsilon \to 0}
            \limsup_{n\to \infty}
            \frac{1}{n} \log N_{\varphi}(n,\varepsilon).
\end{equation}
The topological entropy of $f$ is defined by
$h_{\rm top}(f)=P(f,0)$. We simply write $P_{\rm top}(\varphi)$ and $h_{\rm top}$ if there is no confusion about $f$. The topological pressure is real valued if and only if the topological entropy of $f$ is finite. We use $h_{\rm top}(f)<\infty$ as a standing assumption in this paper.
The topological pressure satisfies the
well-known variational principle
\begin{equation}\label{eqvarpri}
P_{\rm top}(\varphi)=
\sup_{\mu\in \cM} \left\{h_\mu(f)+\int_X \varphi\,d\mu\right\}.
\end{equation}
Here $h_\mu(f)$ denotes the measure-theoretic entropy of $f$ with respect to $\mu$ (see~\cite{Wal:81} for details).
It is a straight forward conclusion that the supremum in~\eqref{eqvarpri} can be replaced by
the supremum taken only over all $\mu\in\cM_{\rm E}$.

%-------------------------------------------------------------------------------------------------------------------------------------------------------
\section{Localized Pressure}
%-------------------------------------------------------------------------------------------------------------------------------------------------------

Our goal is to prove the local version of the variational principle, namely $P_{\rm top}(\varphi,\Phi,w)=P_{\rm m}(\varphi,\Phi,w)$. However, in general this equality does not hold even if the potential $\varphi$ is identically zero. The following examples show that with no additional assumptions we do not have even a one-sided inequality.

\begin{example}\label{ex1} This is an example of a dynamical system where at certain points localized topological pressure is strictly less that the localized measure-theoretic pressure. We concatenate three non-overlapping one-dimensional dynamical systems such that the entropy of the outside components is greater than the entropy of the inside one. We take the potential $\Phi$ to be the identity map and $\varphi$ to be zero. Since in this case the topological pressure does not exceed the topological entropy, the affine property of the measure-theoretic pressure implies the strict inequality at the center points. What follows is the concrete construction.

Let $X=X_1\cup X_2\cup X_3$, where $X_1=[0,1]$, $X_2\subset [2,3]$, and $X_3=[4,5]$.  We define $f:X\to X$ to be the logistic type map on $X_1$ and $X_3$ given by $$f|_{X_1}(x)=4x(1-x),\,\, f|_{X_3}(x)=f|_{X_1}(x-4)+4$$
Then $h_{\rm top}(f|_{X_1})=h_{\rm top}(f|_{X_3})=\log 2$.

Whenever $f|_{X_2}$ satisfies $h_{\rm top}(f|_{X_2})<\log 2$ we will reach our conclusion.  For example, take $X_2$ to be a Cantor set in the interval $[2,3]$ and $f$ to be a homeomorphism on the Cantor set $X_2$ which is topologically conjugate to a subshift whose entropy is strictly less than $\log 2$. One possibility is the subshift with transition matrix $\left(
\begin{array}{cc}
1 & 1 \\
1 & 0 \\
\end{array}
\right)$.
We may also let $X_2=[2,3]$ and $f|_{X_2}(x)=a(x-2)(3-x)$ with $0<a<4$. In this case $h_{\rm top}(f|_{X_2})=0$.

We take the potential $\Phi$ to be the identity map on $X$. Then for any point $w\in \R_{Pt}(\Phi)\cup X_2$ we have $P_{\rm m}(0,\Phi,w)=\log 2$ since localized measure-theoretic pressure is an affine function of $w$. However, $P_{\rm top}(0,\Phi,w)\le h_{\rm top}(f|_{X_2})<\log2$. Therefore, $P_{\rm top}(0,\Phi,w)<P_{\rm m}(0,\Phi,w)$.
\end{example}
The next example will address the reverse inequality.
\begin{example}\label{ex2} Consider a decreasing sequence of disjoint compact intervals $X_n$ on the real line whose left end-points converge to $0$. We define the function $f$ on each $X_n$ to be conjugate to the logistic map $g(x)=4x(1-x)$ on $[0,1]$ and maps $X_n$ onto $X_n$. Moreover, $f(0)=0$. Then $X=\cup_{n=1}^\infty X_n \cup \{0\}$ is compact and $f$ is continuous on $X$. Moreover, for each $n$  the interval $X_n$ is invariant with respect to $f$. Since $f|_{X_n}$ is conjugate to $g(x)=4x(1-x)$ on $[0,1]$, the topological entropy of $f|_{X_n}$ is equal to the topological entropy of $g$ on $[0,1]$ and therefore is $\log 2$.

As an example of such construction consider disjoint dyadic intervals $X_n=[2^{-2n}, 2^{-2n+1}]\,(n\in\mathbb{N})$. In this case $f:X\to X$ is defined in the following way.
 $$f(x)=\left\{
                        \begin{array}{ll}
                          0, & \hbox{if $x=0$;} \\
                          2^n(x-2^{-2n})(2^{-2n+1}-x)+2^{-2n}, & \hbox{if $x\in X_n$.}
                        \end{array}
                      \right.
$$
 Take the identity potential $\Phi:X\to \mathbb{R},\quad \Phi(x)=x$.  Let $\mu_n$ be the entropy maximizing ergodic measures on $X_n$. Then $P_{\rm top}(0,\Phi,\rv(\mu_n))=\log 2$. Since $\rv(\mu_n)\to 0$, we have $P_{\rm top}(0,\Phi,0)=\log 2$. However, $x=0$ is a fixed point of $f$ and also an extreme point of $X$.
Thus, the only invariant measure $\mu$ on $X$ with $\rv(\mu)=0$ is the point-mass measure at zero. Therefore, $P_{\rm m}(0,\Phi,0)=0<P_{\rm top}(0,\Phi,0)$.
\end{example}

We say that $P_{\rm m}(\varphi,\Phi,w)$ is \emph{approximated by ergodic measures} at $w$ if there exists $(\mu_n)_{n\in\mathbb{N}}\subset\cM_E$ such that $\rv(\mu_n)\to w$ and $h_{\mu_n}(f)+\int\varphi\,d\mu_n\to P_{\rm m}(\varphi,\Phi,w)$ as $n\to\infty$. In this case we have $w\in\R_{Pt}(\Phi)$. Indeed, for $r>0$ there exists $n$ such that $\rv(\mu_n)\in D(w,\frac{r}{2})$. The ergodicity of $\mu_n$ implies the existence of $x\in X$ such that $\frac{1}{k}S_k\Phi(x)\in D(\rv(\mu),\frac{r}{2})$ for arbitrary large $k$. Therefore, $\frac{1}{k}S_k\Phi(x)\in D(w,r)$ and thus $w\in \R_{Pt}(\Phi)$.

We say that a metric space is  \emph{Besicovitch} if the Besicovitch covering theorem holds (see  \cite{F,Lo}). The next theorem is a local version of the variational principle.

\begin{theorem}\label{theovar} Let $f:X\to X$ be a continuous map on a compact metric space $X$ that is a Besicovitch space.  Let $\Phi:X\to\bR^m$ and $\varphi:X\to\bR$ be continuous
and let $w\in \R(\Phi)$ such that the map $v\mapsto P_m(\varphi,\Phi, v)$ is continuous at $w$ and $P_{\rm m}(\varphi,\Phi, w)$ is approximated by ergodic measures. Then $P_{\rm top}(\varphi,\Phi, w)=P_{\rm m}(\varphi,\Phi, w)$.
\end{theorem}

\begin{proof}
We first show that $P_{\rm top}(\varphi,\Phi, w)\le P_{\rm m}(\varphi,\Phi, w)$. Fix $\eta>0$. It follows from the definition of $P_{\rm top}(\varphi,\Phi, w)$ and the continuity of $P_{\rm m}(\varphi,\Phi, w)$ that there exist $r>0$ and $\varepsilon>0$ such that
\begin{equation}
\left|\limsup_{n\to\infty}\frac1n\log N_\varphi(n,\varepsilon, w, r)-P_{\rm top}(\varphi,\Phi, w)\right|<\frac{\eta}{2}
\end{equation}
 and for any $v\in D(w,r)\cap \R(\Phi)$ we have
\begin{equation}
\left|P_{\rm m}(\varphi,\Phi, w)-P_{\rm m}(\varphi,\Phi, v)\right|<\frac{\eta}{2}.
\end{equation}
We will now apply the method of constructing measures with large free energies which is commonly used to prove the classical variational principle.
Let $\{F_n\}_{n\in \bN}$ be $(n,\varepsilon)$ separated sets in $X$ such that $\frac1n S_n\Phi(x)\in D(w,r)$ for all $x\in F_n$ and $\sum\limits_{x\in F_n}e^{S_n\varphi(x)}>\frac12 N_{\varphi}(n,\varepsilon, w, r).$ Let $\nu_n$ be the atomic measure concentrated on $F_n$ given by the formula
\begin{equation}
\nu_n=\left(\sum\limits_{x\in F_n}e^{S_n\varphi(x)}\right)^{-1}\sum\limits_{x\in F_n}e^{S_n\varphi(x)}\delta_x,
\end{equation}
where $\delta_x$ denotes the Dirac measure supported on $x$. Consider a sequence of measures $\mu_n=\frac1n\sum\limits_{k=0}^{n-1}\nu_n\circ f^{-k}$ and let $\mu$ be a weak$^*$ accumulation point of $(\mu_n)$. Then (see \cite{Wal:81} or \cite[Section 4.5]{KH}) $\mu$ is $f$-invariant and satisfies
\begin{equation}
\limsup_{n\to\infty}\frac1n\log\sum\limits_{x\in F_n}e^{S_n\varphi(x)}\le h_\mu(f)+\int_X\varphi\,d\mu.
\end{equation}
We conclude that
\begin{equation}\begin{aligned} P_{\rm top}(\varphi,\Phi, w) & \le \limsup_{n\to\infty}\frac1n\log N_\varphi(n,\varepsilon, w, r)+\frac{\eta}{2}\\
&=\limsup_{n\to\infty}\frac1n\log\sum\limits_{x\in F_n}e^{S_n\varphi(x)}+\frac{\eta}{2}\\
&\le P_{\rm m}(\varphi,\Phi,\rv(\mu))+\frac{\eta}{2}.
\end{aligned}\end{equation}
Note that $\rv(\mu)\in D(r,w)$ by the construction of $\mu$. Therefore, $ P_{\rm top}(\varphi,\Phi, w)\le  P_{\rm m}(\varphi,\Phi,w)+\eta$. Since $\eta$ was arbitrary, we obtain the desired inequality $ P_{\rm top}(\varphi,\Phi, w)\le  P_{\rm m}(\varphi,\Phi,w)$.

Now we turn our attention to the opposite inequality. Let $\eta>0$ be arbitrary. As before, we fix $r_0>0$ and $\varepsilon_0>0$ such that for any $0<\varepsilon<\varepsilon_0$
\begin{equation}\label{InitialEst}
\left|\limsup_{n\to\infty}\frac1n\log N_\varphi(n,\varepsilon, w, r_0)-P_{\rm top}(\varphi,\Phi, w)\right|<\frac{\eta}{2}.
\end{equation}
Since $\Phi$ is uniformly continuous on $X$ we may assume that $\varepsilon_0$ is chosen small enough so that for any $n\in \bN$ and $x_1,x_2\in X$ with $d_n(x_1,x_2)<\varepsilon_0$ we have
\begin{equation}\label{UnifCont}
\left|\frac1n S_n\Phi(x_1)-\frac1n S_n\Phi(x_2)\right|\le \frac{r_0}{3}.
\end{equation}
Since $P_{\rm m}(\varphi, \Phi, w)$ is approximated by ergodic measures, there exists $\mu\in\cM_E$ such that
\begin{equation}\label{ErgMeas}
|\rv(\mu)-w|<\frac{r_0}{3}\quad \text{and}\quad P_{\rm m}(\varphi,\Phi,w)-\frac{\eta}{4}<h_\mu(f)+\int\varphi\, d\mu.
\end{equation}
There is a generalization of Katok's characterization of the measure-theoretic entropy in terms of ergodic measures to the concept of topological pressure derived in \cite{He}. See \cite{Kat:80} for the original approach. We are using the following set up:  Fix $0<\delta<1$. We say that $E$ is an $(n,\varepsilon)$-spanning set for $Y\subset X$ if $Y\subset\cup_{x\in E}B_n(x,\varepsilon)$.  Denote by $Q_\varphi(n,\varepsilon, \mu, \delta)=\inf\left\{\sum_{x\in E}e^{S_n\varphi(x)}\right\}$, where the infimum is taken over all $(n,\varepsilon)$-spanning sets $E$ of a set of $\mu$-measure more than or equal to $1-\delta$. Then
\begin{equation}\label{He}
h_\mu(f)+\int\varphi\, d\mu=\lim_{\varepsilon\to 0}\liminf_{n\to\infty}\frac{1}{n}\log Q_\varphi(n, \varepsilon, \mu, \delta).
\end{equation}
There exists a decreasing sequence of strictly positive numbers $\varepsilon_i<\varepsilon_0,\,(i\in \bN)$ with $\lim_{i\to 0}\epsilon_i=0$ and corresponding sequences of $(n,\varepsilon_i)$-spanning sets $E_n(\varepsilon_i)\,(n\in\bN)$ such that
\begin{equation}
\sum_{x\in E_n(\varepsilon_i)}e^{S_n\varphi(x)}<2Q_\varphi(n,\varepsilon_i, \mu, \delta)
\end{equation}
  and
\begin{equation}\label{UpperPr}\left(h_\mu(f)+\int\varphi\, d\mu\right)-\frac{\eta}{4}<\liminf_{n\to\infty}\frac1n\log\sum_{x\in E_n(\varepsilon_i)}e^{S_n\varphi(x)}.
\end{equation}
We may assume that each $E_n(\varepsilon_i)$ is a minimal spanning set with respect to the inclusion.
Since $\mu$ is ergodic, the basin of $\mu$  defined by
\begin{equation}
\cB(\mu)=\left\{x\in X : \frac1n \sum_{k=1}^{n-1}\delta_{f^k(x)}\to \mu \text{ as }n\to\infty\right\}
\end{equation}
 is a set of full $\mu$-measure by Birkhoff's Ergodic Theorem. We define
\begin{equation}
\cB_{n,\frac{r_0}{3}}(\mu)=\left\{x\in \cB(\mu): \left|\frac{1}{l}S_l\Phi(x)-\rv(\mu)\right|<\frac{r_0}{3}\text{ for all }l\ge n\right\}.
\end{equation}
Since $(\cB_{n,\frac{r_0}{3}}(\mu))_{n\in \bN}$ is an increasing sequence of Borel sets whose union is a set of full $\mu$-measure, we conclude that $\lim\limits_{n\to \infty}\mu(\cB_{n,\frac{r_0}{3}}(\mu))=1$. Consider the sequence of sets
\begin{equation}
\tilde{E}_n(\varepsilon_i)=\left\{x\in E_n(\varepsilon_i): B_n(x,\varepsilon_i)\cap\cB_{n,\frac{r_0}{3}}(\mu)\ne\emptyset\right\}.
\end{equation}
It follows from (\ref{UnifCont}) and (\ref{ErgMeas}) that for any $x\in \tilde{E}_n(\epsilon_i)$ we have $\frac1n S_n\Phi(x)\in D(w,r_0)$.
When $n$ is sufficiently large, $\tilde{E}_n(\varepsilon_i)$ is a spanning set for a set of $\mu$-measure greater than $1-\delta'$ where $\delta<\delta'<1$. Therefore,
\begin{equation}
Q_{\varphi}(n, \varepsilon_i,\mu, \delta')\le\sum_{x\in \tilde{E}_n(\varepsilon_i)}e^{S_n\varphi(x)}\le \sum_{x\in E_n(\varepsilon_i)}e^{S_n\varphi(x)}<2Q_\varphi(n,\varepsilon_i,\mu,\delta).
\end{equation}
It follows from the fact that (\ref{He}) also holds for $\delta'$ that  for $\varepsilon_i$ small enough (\ref{UpperPr}) remains true when we replace $E_n(\varepsilon_i)$ by $\tilde{E}_n(\varepsilon_i)$. %Denote such $\varepsilon_i$ by $\tilde{\varepsilon}$.

Let $\beta$ be a Besicovitch constant of $X$. Note that this constant can be chosen independently of the metrics $d_n$ since they are decreasing in $n$. It follows from the Besicovitch covering theorem and the fact that $\tilde{E}_n(\varepsilon_i)$ is minimal that $\tilde{E}_n(\varepsilon_i)=\cup_{k=1}^\beta F_n^k(\varepsilon_i)$ where each $F_n^k(\varepsilon_i)$ is an $(n,\varepsilon_i)$-separated set. We conclude that
\begin{equation}\label{eqadd1}
 \sum_{x\in \tilde{E}_n(\varepsilon_i)}e^{S_n\varphi(x)}\le\beta\sup_{1\le k\le n}\left\{\sum_{x\in F_n^k(\varepsilon_i)}e^{S_n\varphi(x)}\right\}\le \beta N_\varphi(n,\varepsilon_i, w, r_0).
\end{equation}
Combining inequality \eqref{eqadd1} with (\ref{InitialEst}), (\ref{ErgMeas}) and (\ref{UpperPr}) we obtain
\begin{equation}
P_{\rm m}(\varphi,\Phi,w)-\frac{\eta}{2}<P_{\rm top}(\varphi,\Phi,w)+\frac{\eta}{2}.
\end{equation}
 Since $\eta$ was arbitrary, this concludes the proof of the theorem.
\end{proof}
Note that the left hand side inequality between the topological and measure-theoretic localized pressures was proven under milder assumptions. More precisely, we have the following.
\begin{corollary}
Let $f:X\to X$ be a continuous map on a compact metric space $X$, let $\Phi:X\to\bR^m$ and $\varphi:X\to\bR$ be continuous
and let $w\in \R_{Pt}(\Phi)$ such that the map $v\mapsto P_m(\varphi,\Phi, v)$ is continuous at $w$. Then $P_{\rm top}(\varphi,\Phi, w)\le P_{\rm m}(\varphi,\Phi, w)$.
\end{corollary}
\begin{remarks}
 {\rm (i) }Note that whenever $w\in \R_{Pt}(\Phi)\cap X_2$ in Example \ref{ex1}  then $P_{\rm m}(0,\Phi,w)$ cannot be approximated by ergodic measures. Indeed, for any $\mu\in \cM_E$ with $w=\rv(\mu)\in X_2$ we have $\mu(X_1)=\mu(X_3)=0$; thus, $h_\mu(f)\le h_{\rm top}(f|_{X_2})<\log 2$ follows from the variational principle.\\
 {\rm (ii) } In Example \ref{ex2}  we  observe that the function $w\mapsto P_m(0,\Phi, w)$ is not continuous at $w=0$. We have $\rv(\mu_n)\to 0, \,\,P_m(0,\Phi, \rv(\mu_n))=\log 2$ and $P_m(0,\Phi, 0)=0$.

\end{remarks}

\section{equilibrium states}
Let $f:X\to X$ be a continuous map on a compact metric space
and let $\Phi=(\phi_1,\ldots,\phi_m)\in C(X,\bR^m)$.  Fix $w\in\R(\Phi)$.
We recall the definition of  $\mu\in \cM_\Phi(w)$ being a localized equilibrium state of $\varphi\in C(X,\bR)$ with respect to $\Phi$ and $w$ in \eqref{defeqsta}.

We say that the entropy map is upper semi-continuous at $w\in \R(\Phi)$ if for every $(\mu_n)_n\subset \cM$ with $\rv(\mu_n)\to w$ and every accumulation point $\mu$ of $(\mu_n)_n$ we have $\limsup_{n\to \infty} h_{\mu_n}(f)\leq h_{\mu}(f)$.
Note that if the entropy map is upper semi-continuous at $w$ then there exists for each $\varphi\in C(X,\bR)$ at least one localized equilibrium state of $\varphi$. The following example shows that the existence of a localized equilibrium state does in general not imply the existence of an ergodic localized equilibrium state. This differs from the theory of classical equilibrium states where the existence of an equilibrium state always guarantees the existence of an ergodic equilibrium state (see \cite{Wal:81}).

\begin{example}\label{ex3}
Let $a,b,c,d\in \bR$ with $a<b<c<d$. Let $X=[a,b]\cup [c,d]$ and $f:X\to X$ be a continuous transformation with an upper semi-continuous entropy map $\mu\mapsto h_\mu(f)$ such that $f([a,b])\subset [a,b]$ and $f([c,d])\subset [c,d]$.
Moreover, we assume that $f(a)=a$, $f(d)=d$, and $h_{\rm top}(f|_{[a,b]})=h_{\rm top}(f|_{[c,d]})\ne 0$. Consider the potentials $\Phi={\rm id}_X$ and $\varphi\equiv 0$. Since $\delta_a,\delta_d\in \cM$ the  convexity of $\R(\Phi)$ implies $\R(\Phi)=[a,d]$. Any $w\in (b,c)$ can be written as $w=\alpha \rv(\mu_1)+(1-\alpha)\rv(\mu_2)$, where $\alpha\in(0,1)$ and $\mu_1,\mu_2$ are ergodic entropy maximizing measures on $[a,b]$ and $[c,d]$ respectively. It follows that the measure $\mu=\alpha\mu_1+(1-\alpha)\mu_2$ is a localized equilibrium state of $\varphi$ with respect to $\Phi$ and $w$. However, the set $\cM_\Phi(w)$ does not contain any ergodic measure.
\end{example}

We will see that even in the case of systems satisfying the strongest possible existence and uniqueness results for
classical equilibrium states the situation for localized equilibrium states is rather different.
We now introduce the class of systems with strong thermodynamic properties.

\subsection{Systems with strong thermodynamic properties}

 We say $f:X\to X$ has strong thermodynamic properties (which we abbreviate by \STP) if the following conditions hold:
\begin{enumerate}
\item $h_{\rm top}(f)<\infty$;
\item  The entropy map $\mu\mapsto h_\mu (f)$ is upper semi-continuous;
\item  The map $\varphi\mapsto  P_{\rm top}(f,\varphi)$ is real-analytic on $C^\alpha(X,\bR)$;
\item \label{ref4} Each  potential $\varphi \in C^\alpha(X,\bR)$ has a
unique equilibrium measure $\mu_\varphi$ such that $P(\varphi)=h_{\mu_{\varphi}}(f)+\int\varphi\,d\mu_{\varphi}$. Furthermore,
$\mu_\varphi$ is ergodic and given $\psi\in C^\alpha(X,\bR)$ we have
\begin{equation}\label{eqdifpre}
\frac{d}{dt} P_{\rm top}(f,\varphi + t\psi )\Big|_{t=0}= \int_X \psi
\,d\mu_\varphi.
\end{equation}
\item  For each $\varphi$, $\psi\in C^\alpha(X,\bR)$ we have
$\mu_\varphi=\mu_\psi$ if and only if $\varphi-\psi$ is cohomologous
to a constant.
\item \label{ref5} For each $\varphi$, $\psi\in C^\alpha(X,\bR)$ and
$t\in\bR$ we have
\begin{equation}\label{gg33}
\frac{d^2}{dt^2} P_{\rm top}(f,\varphi + t\psi )\ge 0,
\end{equation}
with equality if and only if $\psi$ is cohomologous to a constant.

\end{enumerate}
Note that for several classes of systems properties (3)-(6) hold even for a wider class of potentials, namely for potentials with summable variation (see for example \cite{Je}). For simplicity, we restrict our
considerations to H\"older continuous potentials.

Some examples of systems with strong thermodynamic properties are expansive homeomorphisms with specification which include topological mixing two-sided subshifts of finite type as well as diffeomorphisms with a locally maximal topological mixing hyperbolic set, see \cite{Bo2,HRu, KH, Ru}. We note that in all these examples the measure $\mu_\varphi$ in property (4) is a Gibbs measure.
Next, we introduce some concepts about shift maps  that will be used later on.

Let $d\in \bN$ and let $\cA=\{0,\cdots,d-1\}$ be a finite alphabet in $d$ symbols. The (one-sided) shift space $X$ on the alphabet $\cA$ is the set of
all sequences $x=(x_n)_{n=1}^\infty$ where $x_n\in \cA$ for all $n\in \bN$.  We endow $X$ with the Tychonov product topology
which makes $X$ a compact metrizable space. For example, given $0<\alpha<1$ it is easy to see that
\begin{equation}\label{defmet}
d(x,y)=d_\alpha(x,y)=\alpha^{\inf\{n\in \bN:\  x_n\not=y_n\}}
\end{equation}
defines a metric which induces the Tychonov product topology on $X$.
The shift map $f:X\to X$ (defined by $f(x)_n=x_{n+1}$) is a continuous $d$ to $1$ map on $X$.
If $Y\subset X$ is an $f$-invariant set we  say  that $f|_Y$ is a sub-shift. In particular, for  a $d\times d$ matrix $A$ with values in $\{0,1\}$
we define $X_A=\{x\in X: A_{x_n,x_{n+1}}=1\}$. It is easy to see that $X_A$ is a closed (and therefore compact) $f$-invariant set and we say that $f|_{X_A}$ is a subshift of finite type. A subshift of finite type is (topologically) mixing if $A$ is aperiodic, that is, if there exists $n\in \bN$ such that $A^n_{i,j}>0$ for all $i,j\in \cA$.

Analogously, we obtain the concept of two-sided shift spaces and shift maps by defining $X$ to be the space of all bi-infinite sequences $x=(x_n)_{n=-\infty}^\infty$ where $x_n\in \cA$ for all $n\in \bZ$. It is a well-known fact that topological mixing sub-shifts of finite type have strong thermodynamic properties (see \cite{Ru}).

\subsection{Interior and boundary equilibrium states}
From now on we assume that $f$ has strong thermodynamic properties, $\Phi: X\to\bR^m$ and $\dim \R(\Phi)=m$.  Recall that if $\Phi$ is H\"older continuous then  $\dim \R(\Phi)=m$  is equivalent to the condition that no non-trivial linear
combination $t\cdot \Phi= t_1\phi_1+ \ldots + t_m \phi_m$ is cohomologous to a constant. For $A\subset \bR^l$ we define the relative interior of $A$ (denoted by ${\rm ri }\ A$) as the interior of $A$ considered as a subset of the smallest affine subspace of $\bR^l$  containing $A$. In particular, if $A$ has non-empty interior then the relative interior and the interior of $A$ coincide.

\begin{definition}\label{defintbo} Suppose $\mu\in \cM_\Phi(w)$ is a localized equilibrium state of $\varphi\in C(X,\bR)$ with respect to $\Phi$ and $w$. We say that $\mu$ is an interior equilibrium state if $(\int \varphi d\mu,w)\in {\rm ri}\ \R(\varphi,\Phi)$.  Otherwise, call $\mu$
a localized  equilibrium state at the boundary.
\end{definition}

We note that  $\dim  \R(\varphi,\Phi)=m$ if and only if either $\varphi$ is cohomolous to a constant or $\varphi$ is cohomologous to some nontrivial linear combination of $\Phi$. In this situation we say that a localized equilibrium state of $\varphi$ with respect to $\Phi$ and $w$ is a localized measure of maximal entropy at $ w$.

The following example shows that localized equilibrium states at the boundary are in general not unique.

\begin{example}\label{ex4}
Let $f:X\to X$ be the one-sided full shift  with alphabet $\{0,1,2,3\}$.
Let $C$ be a compact and convex subset of $\bR^2$ whose boundary $\partial C$ is a strictly convex Jordan curve. Pick any point $w_\infty\in\partial C$. Then there exists a line passing through $w_\infty$ which does not intersect ${\rm int}\, C$, but its orthogonal line does. Let $w_0$ be any point in ${\rm int}\, C$ on that orthogonal line and let $v_1,v_2$ be points on $\partial C$ on opposite sides with respect to the line. Denote by $l_1, l_2$ the arcs in $\partial C$ joining $v_1, v_2$ and $w_\infty$. For $i=1,2$ we pick a strictly unidirectional sequence $(v_i(k))_{k\in \bN}\subset l_i$ starting at $v_i$ and going towards $w_\infty$.  We require that $v_i(1)=v_i$ and  $|v_i(k)-w_\infty|<1/2^k$ for all $k>1$, in particular $\lim_{k\to \infty} v_i(k)=w_\infty$.

Next, we define several subsets of $X$.
Let $S_1=\{0,1\}, S_{2}=\{2,3\}$ and fix $ \alpha\in \bN, \alpha\geq 3$.
For $i=1,2$ and all $k\geq \alpha$ we define $Y_i(k)=\{x \in X: x_1,\ldots, x_{k}\in S_i\}$. Moreover, let $Y_0(\alpha)= X\setminus (Y_1(\alpha)\cup  Y_{2}(\alpha))$.
\\[0.2cm]
\noindent
Finally, we define a potential $\Phi: X\to \bR^2$ by
\begin{equation}\label{defpotphi}
\Phi(x)=\begin{cases}
w_0\qquad & {\rm if}\,\,   x\in Y_0(\alpha)\\
                         v_{i}(k-\alpha) & {\rm if}\,\,  x\in Y_{i}(k-1)\ {\rm and}\ x\not\in Y_i(k),\, k> \alpha\\
                          w_\infty & {\rm if}\,\,   \  x\in Y_{i}(k)\  {\rm for\, all }\, k\, {\rm for\ some}\ i\in \{1,2\}
            \end{cases}
\end{equation}
Note that $\Phi(x)=w_\infty$ if and only if either $x_k\in \{0,1\}$ for all $k\in \bN$ or $x_k\in \{2,3\}$ for all $k\in \bN$, in particular $f|_{\Phi^{-1}(w_\infty)}$ is a subshift finite type $f_{\rm A}$
with transition matrix \begin{equation}\label{matA}
 {\rm A} = \left| \begin{array}{cccc}
1& 1 & 0 & 0 \\
1&1&0&0  \\
0&0&1&1\\
0&0&1&1\end{array} \right|.
\end{equation}
\end{example}

To illustrate this example we consider a case where the set $C$ and the sets of points $v_1(k)$, $v_2(k)$ are symmetric about the line through $w_{\infty}$ and $w_0$. We denote by $w_i(j)$ the rotation vectors of the periodic orbits of length $j$ whose generators have the first $j-1$ coordinates in  $S_i$ and the $j^{\text{th}}$ coordinate in the complementary alphabet $S_{3-i}$. Precisely, for $j>\alpha$ and $i=1,2$ we have

\begin{equation}\label{wj} w_i(j)=\frac{\sum\limits_{k=1}^{j-\alpha}v_i(k)+\alpha w_0}{j}.\end{equation}

We show that in this case the boundary of $\R(\Phi)$ is the infinite polygon. Moreover, there is a neighborhood of $w_\infty$ where the vertices of $\R(\Phi)$ are exactly $w_i(j)$, $i=1,2$ and $j>j_0$ for some integer $j_0$ which depends on properties $\partial C$. We prove this fact in the next proposition by introducing a new coordinate system with the origin at $w_\infty$ and and the $x$-axis passing through $w_0$. For a given point $a\in\bR^2$ we write ${\rm pr}_x(a)$ and ${\rm pr}_y(a)$ for the $x$ and $y$ coordinates respectively.  Let $y=l(x)$ denote the parametrized boundary curve of the upper half of $C$. By symmetry $y=-l(x)$ coincides with the boundary curve of the lower half. For simplicity we add an additional assumption on $l(x)$ that guaranties that $j_0=\alpha$.

\begin{proposition}\label{prop1}
Suppose that $\{x_k\}\subset [0,1]$ is a decreasing sequence such that $x_k\le\frac{1}{2^{k}}$, $l:[0,1]\mapsto \bR$ is an increasing and strictly convex function such that $l(0)=0$ and $l(x_1)>(\alpha+1)l(x_2)$. Let $w_0$ be the midpoint between $x_1$ and $x_2$. For $i=1,2$ denote by $v_i(k)=(x_k, (-1)^il(x_k))$, let $w_i(j)$ be as in (\ref{wj}) for $j>\alpha$ and set $w_i(\alpha)=\left(\frac{(\alpha-1){\rm pr}_x(w_0)+x_1}{\alpha},(-1)^i\frac{l(x_1)}{3\alpha}\right)$. Then for the potential $\Phi$ defined in Example \ref{ex4} we have 
\begin{equation}
\R(\Phi)=\overline{\rm Conv}\{w_i(j): j\ge\alpha,\,i=1,2\}.
\end{equation}
\end{proposition}
\begin{proof}
First we show that the sequence of points $\{w_1(j)\}_{j>\alpha}$ is monotonically decreases to the origin. By symmetry, this  immediately implies  that the sequence $\{w_2(j)\}_{j>\alpha}$  increases monotonically  to the origin. It follows from (\ref{wj}) that for any $j>\alpha$ we have
\begin{align}
w_1(j)-w_1(j+1)
%&=\frac{1}{j(j+1)}\left[(j+1)\sum\limits_{k=1}^{j-\alpha}v_i(k)+(j+1)\alpha w_0-j\sum\limits_{k=1}^{j-\alpha+1}v_i(k)-j\alpha w_0\right]\\
&=\frac{1}{j(j+1)}\left[\alpha w_0+\sum\limits_{k=1}^{j-\alpha}v_1(k)-jv_1(j+1-\alpha)\right].
\end{align}
The $x$-coordinate of $w_1(j)-w_1(j+1)$ is always positive, since the $x_k$ are decreasing and ${\rm pr}_x(w_0)>x_{j+1-\alpha}$. The $y$-coordinate of ${w_1(j)-w_1(j+1)}$ simplifies to 
\begin{equation}
\sum_{k=1}^{j-\alpha}l(x_k)-jl(x_{j+1-\alpha}).
\end{equation}
 This expression is positive whenever $l(x_1)>(\alpha+1)l(x_{j+1-\alpha})$. This can always be achieved starting from some $j_0$ since $l(x_k)$ is decreasing to zero. Therefore, $w_1(j)$ are decreasing for $j>j_0$. The assumptions of the proposition assure that we may take $j_0=\alpha$, however this condition is not essential.

The result by Sigmund that the periodic point measures are dense in $\cM$ reduces our considerations to rotation vectors of periodic orbits.

Suppose $x\in X$ is a periodic point of period $n$. We may assume that $x=(\xi_1,...,\xi_n,...)$ and $(\xi_1,...\xi_n)$ is maximally partitioned into $k$ blocks of sizes $n_1,...,n_k$ such that $n_1+...+n_k=n$, and each block exclusively contains elements of either $S_1$ or $S_2$. It follows from the construction of $\Phi$ that $n\cdot\rv(x)$ (where $\rv(x)$ denotes the rotation vector of the unique invariant measure supported on the orbit of $x$) is the sum of blocks of vectors of the form
\begin{equation}\label{decomposition}
(\alpha-1)w_0+\sum_{i=1}^{n_j-(\alpha-1)}v_s(i).
\end{equation}
Here $s=1$ if the elements of $j^\text{th}$ block are from $S_1$ and $s=2$ if the elements of $j^\text{th}$ block are from $S_2$.
In case $n_j\le \alpha-1$ the block's contribution is $n_jw_0$.

First we show that $\rv(x)\in \overline{\rm Conv}\{w_s(j): j\ge\alpha,\,s=1,2\}$ for $k=2$. In this case we have
\begin{equation}\rv(x)=\frac1n \left[2(\alpha-1)w_0+\sum\limits_{i=1}^{n_1-\alpha+1}v_1(i)+\sum\limits_{i=1}^{n_2-\alpha+1}v_2(i)\right].\end{equation}
By symmetry we restrict ourselves to the case $n_1>n_2$. We compare $\rv(x)$ with points $w_1(n)$ and $\frac{n_1}{n}w_1(n_1)+\frac{n_2}{n}w_1(n_2)$. For the $x$-coordinates we obtain that
\begin{equation}{\rm pr}_x(w_1(n))\le{\rm pr}_x\left(\rv(x)\right)\le{\rm pr}_x\left(\frac{n_1}{n}w_1(n_1)+\frac{n_2}{n}w_1(n_2)\right)\end{equation}
as long as $n_j>\alpha$ for either $j=1$ or $j=2$. When $n_1=n_2=\alpha$ we obtain $\rv(x)=\frac{(\alpha-1)w_0+(x_1,0)}{\alpha}$, which is the mid-point between $w_1(\alpha)$ and $w_2(\alpha)$. For the $y$-coordinates we obtain
\begin{equation}{\rm pr}_y\left(\frac{n_1}{n}w_1(n_1)+\frac{n_2}{n}w_1(n_2)\right)\ge{\rm pr}_y(w_n)>{\rm pr}_y(\rv(x)),\end{equation} and thus $\rv(x)\in \overline{\rm Conv}\{w_i(j): j\ge\alpha,\,i=1,2\}$.

The case $k=3$ is similar. We have $n=n_1+n_2+n_3$ with $n_1\ge n_3$. By symmetry, we may assume that $\rv(x)$ is above the $x$-axis and that we can write
\begin{equation}\rv(x)=\frac1n \left[3(\alpha-1)w_0+\sum\limits_{i=1}^{n_1-\alpha+1}v_1(i)+\sum\limits_{i=1}^{n_2-\alpha+1}v_2(i)+\sum\limits_{i=1}^{n_3-\alpha+1}v_1(i)\right].\end{equation}  We compare $\rv(x)$ with points $w_1(n)$ and $\frac{n_1+n_2}{n}w_1(n_1+n_2)+\frac{n_3}{n}w_1(n_3)$. For the $x$-coordinates we obtain
\begin{equation}{\rm pr}_x(w_1(n))\le{\rm pr}_x(\rv(x))\le{\rm pr}_x\left(\frac{n_1+n_2}{n}w_1(n_1+n_2)+\frac{n_3}{n}w_1(n_3)\right)\end{equation} whenever $n_j>\alpha$ for at least for one $j$. In the case $n_1=n_2=n_3=\alpha$ we have $\rv(x)=w_1(\alpha)$.

For the $y$-coordinates we see that
\begin{equation}{\rm pr}_y\left(\frac{n_1+n_2}{n}w_1(n_1+n_2)+\frac{n_3}{n}w_1(n_3)\right)\ge{\rm pr}_y(w_n)>{\rm pr}_y(\rv(x)).\end{equation} It follows that $\rv(x)\in \overline{\rm Conv}\{w_i(j): j\ge\alpha,\,i=1,2\}$.

To conclude the proof we notice that the rotation vector of any periodic orbit can be written as a convex combination of vectors described in the previous two cases and $w_0$.
\end{proof}
Figure 1 illustrates the rotation set of the potential $\Phi$ (see \eqref{defpotphi}) where the set $C$ is a ellipse $(x-1)^2+\frac{y^2}{2^2}=1$, $x_1=1$, $x_k=\frac{1}{6^{k}}$ for $k>1$ and $\alpha=3$. Below we plot 1000 data points of this rotation set. The shape of the resulting graph gives it the name 'fish'.
\begin{figure}[h]
\begin{center}
\scalebox{0.65}{\includegraphics{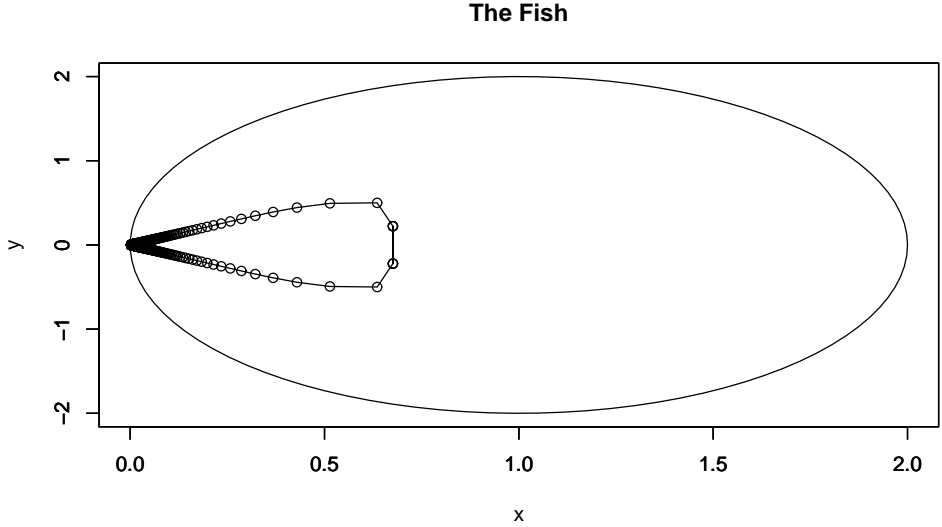}}
\caption{The Rotation set of the fish based on 1000 data points.}
\label{RotSet}
\end{center}
\end{figure}

We now  list several properties of the system in Example \ref{ex4} that hold without the symmetry assumptions in Proposition \ref{prop1}.
\begin{theorem}\label{theomex}
Let $X,f$ and $\Phi$ be as in Example \ref{ex4}. Then
\begin{enumerate}
\item[(i)] $\Phi$ is Lipschitz continuous;
\item[(ii)] ${\rm int}\ \R(\Phi)\not=\emptyset$ and $\R(\Phi)\subset {\rm int}\  C\cup \{w_\infty\}$;
\item[(iii)]  There exists precisely two ergodic localized measures of maximal entropy at $w_{\infty}$;
\item[(iv)]  The entropy function $w\mapsto H(w)\eqdef P_{\rm top}(0,\Phi,w)$ is real analytic in ${\rm int}\ \R(\Phi)$.
\end{enumerate}
\end{theorem}
\begin{proof}
{\rm (i) } We will work with the $d_{1/2}$ metric (see \eqref{defmet}) on $X$ to show that $\Phi$ is Lipschitz continuous. Set $\gamma={\rm diam}(C)$. Let $x,y\in X$ with $\Phi(x)\not=\Phi(y)$. If $\Phi(x)=w_0$
then $x_k\not=y_k$ for some $k\leq \alpha$. Hence,
\begin{equation}\label{eqwsx}
\|\Phi(x)-\Phi(y)\|_2\leq \gamma = \gamma 2^\alpha \frac{1}{2^\alpha} \leq \gamma 2^\alpha d(x,y).
\end{equation}
The case $\Phi(y)=w_0$ is analogous. The case $\Phi(x)\in l_i\setminus\{w_\infty\}$ and $\Phi(y)\in l_j\setminus \{w_\infty\}$ with $i\not=j$ can be treated analogously as in \eqref{eqwsx}.
It remains to consider the case $\Phi(x),\Phi(y)\in l_i$ for some $i=1,2$. Without loss of generality we assume that $\Phi(y)$ is further along on the path to $w_\infty$ as $\Phi(x)$.
Thus, $d(x,y)\geq\frac{1}{2^{k_x+1}}$ where $k_x$ is defined by $\Phi(x)=w_i(k_x)$. Since $|w_i(k)-w_\infty|<1/2^k$ for all $k \in \bN$, we conclude that
\begin{equation}\label{eqwsx1}
\|\Phi(x)-\Phi(y)\|_2\leq \frac{2}{2^{k_x}} \leq 4 d(x,y)
\end{equation}
which completes the proof of (i).\\
\noindent
(ii) Note that points $w_0, w_\infty$ and $\frac{v_1(1)+\alpha w_0}{\alpha+1}$ belong to  $\R(\Phi)$ and thus ${\rm int }\,\R(\Phi)\not=\emptyset$.

To prove that $\R(\Phi)\subset {\rm int}\  C\cup \{w_\infty\}$ we apply again the result of Sigmund that the periodic point measures are weak$^\ast$ dense in $\cM$. Suppose $x\in X$ is a periodic point of period $n$ and that we have the decomposition (\ref{decomposition}). We use the notation
\begin{equation}w_s^*(j)=\frac{(\alpha-1)w_0+\sum\limits_{i=1}^{j-(\alpha-1)}v_s(i)}{j},\quad s=1,2;\,\, j\ge\alpha.\end{equation}
Then the rotation vector of any periodic orbit can be written as a convex combination of $w_s^*(j)$ and $w_0$. The fact that the set $C$ is strictly convex implies that $w_s^*(j)\in \text{int }C$ for all $j$. Also, it is easy to see that the $w_s^*(j)$ converge to $w_\infty$ as $j\to\infty$. Indeed,
\begin{equation}\begin{aligned}
|w_\infty-w_s^*(j)|&\le\frac1j\left((\alpha-1)|w_0-w_\infty|+\sum\limits_{i=1}^{j-\alpha+1}|v_s(i)-w_\infty|\right)\\
&\le \frac{1}{j}\left((\alpha-1)|w_0-w_\infty|+\sum\limits_{i=1}^{j-\alpha+1}\frac{1}{2^i}\right)\\
&\le\frac{(\alpha-1)|w_0-w_\infty|+1}{j}.
\end{aligned}\end{equation}
Since $\{w_i^*(j)\}_{j\ge\alpha}\subset {\rm int}\,C$ and $w_\infty$ is their only accumulation point, we have $\overline{\rm Conv}\{w_i^*(j)\}_{j\ge\alpha}\subset {\rm int}\,C\cup\{w_\infty\}$ and thus $\R(\Phi)\subset {\rm int}\,C\cup\{w_\infty\}$.
\\
\noindent
(iii) We will compute the logarithmic rate of growth of periodic orbits with rotation vectors in the neighborhood of $w_\infty$. Fix $0<r<\frac12 d(w_0,w_\infty)$. Suppose $x\in X$ is a periodic point of period $n$ and $\rv(x)\in D(w_\infty,r)$. We may assume decomposition (\ref{decomposition}). Since there are $k$ blocks and each block contributes at least one $w_0$ to $\rv(x)$ we have \begin{equation}\frac{kd(w_0,w_\infty)}{n}<d(\rv(x),w_\infty)<r.\end{equation} Therefore, $k<\frac{nr}{d(w_0,w_\infty)}$. Denote $m=\lfloor{\frac{nr}{d(w_0,w_\infty)}}\rfloor$, the largest integer smaller than $\frac{nr}{d(w_0,w_\infty)}$. Note that $m<\frac12 n$ since $r<\frac12 d(w_0,w_\infty)$.

The maximal number of points of period $n$ in $D(r,w_\infty)$ is
\begin{equation}\sum_{k=1}^m\binom{n}{k-1}\prod_{j=1}^k 2^{n_j}=2^n\sum_{k=1}^m\binom{n}{k-1}.\end{equation}

We will estimate $\sum_{k=1}^m\binom{n}{k-1}=\binom{n}{m-1}+\binom{n}{m-2}+\ldots+\binom{n}{0}$. We have
\begin{equation}\begin{aligned}
&\frac{\binom{n}{m-1}+\binom{n}{m-2}+\ldots+\binom{n}{0}}{\binom{n}{m}}\\
&=\frac{m}{n-m+1}+\frac{m(m-1)}{(n-m+1)(n-m+2)}+\ldots\\%\frac{m(m-1)(m-2)}{(n-m+1)(n-m+2)(n-m+3)}+\ldots\\
&\le \frac{m}{n-m+1}+\left(\frac{m}{n-m+1}\right)^2+\left(\frac{m}{n-m+1}\right)^3+\ldots\\
&=\frac{m}{n-2m+1}.
\end{aligned}\end{equation}
In the last equality we used the sum of geometric progression with common ratio $\frac{m}{n-m+1}$ which is less than one since $m<\frac12 n$. We obtain \begin{equation}\sum_{k=1}^m\binom{n}{k-1}\le\binom{n}{m}\frac{m}{n-2m+1}.\end{equation} Using the well known fact that $\frac{n^n}{e^{n-1}}\le n!\le\frac{(n+1)^{n+1}}{e^n}$ we obtain
\begin{equation}\log\binom{n}{m}\le(n+1)\log(n+1)-m\log m-(n-m)\log(n-m).\end{equation}
To simplify the notation in the following computation we denote $\rho=\frac{r}{d(w_0,w_\infty)}$. Using $n\rho-1< m\le n\rho $, we estimate the growth rate of the periodic orbits of period $n$ in $D(r,w_\infty)$.
\begin{multline}
\frac{1}{n}\log 2^n\sum_{k=1}^m\binom{n}{k-1}\le \log 2+\frac{1}{n}\log\frac{n\rho}{n-3n\rho+3}+\frac{n+1}{n}\log(n+1)\\ -\frac{n\rho-1}{n}\log(n\rho-1)-\frac{n-n\rho}{n}\log(n-n\rho). \\
\end{multline}
Passing to the limit as $n$ approaches infinity we obtain the growth rate of the periodic orbits
\begin{equation}\begin{aligned}
  &\limsup\limits_{n\to\infty}\frac{1}{n}\log 2^n\sum_{k=1}^m\binom{n}{k-1}\\
  &\le \log 2+\lim_{n\to\infty}\left[\log(n+1)-\rho\log(n\rho-1)-(1-\rho)\log(n-n\rho)\right] \\
  & = \log 2 + \log\rho -(1-\rho)\log(1-\rho).
\end{aligned}\end{equation}
Note that the last expression is greater than $\log 2$ since $\log(1-\rho)<0$. Since $\rho\to 0$ as $r\to 0$, we have \begin{equation}\lim_{r\to 0}\left[\log 2 + \log\rho -(1-\rho)\log(1-\rho)\right]=\log 2.\end{equation}

Therefore, we have $P_{\rm top}(0,\Phi,w_\infty)=h_{\rm top}(f_{\rm A})=\log 2$. Thus, the two distinct ergodic localized measures of maximal entropy at $w_\infty$ are the two ergodic measures of maximal entopy of $f_{\rm A}$.\\
\noindent
Finally, (vi) is a result of \cite{KW}.
\end{proof}

\begin{remarks}
{\rm (i) }The cardinality of ergodic localized equilibrium states at the boundary is in general not preserved under small changes of the potential $\Phi$. Indeed, for any $\epsilon>0$ pick a point $w_\epsilon\in\partial C$ such that $w_\epsilon\ne w_\infty$ and dist$(w_\infty,w_\epsilon)<\epsilon$. Redefine $v_2(k)$ within an $\epsilon$-neighbourhood so that $|w_\epsilon-v_2(k)|<\frac{1}{2^k}$. Define
\begin{equation}
\Phi_\epsilon(x)=\begin{cases}
w_0\qquad & {\rm if}\,\,   x\in Y_0(\alpha)\\
                         v_{i}(k-\alpha) & {\rm if}\,\,  x\in Y_{i}(k-1)\ {\rm and}\ x\not\in Y_i(k),\, k> \alpha\\
                          w_\infty & {\rm if}\,\,   \  x\in Y_{1}(k)\  {\rm for\, all }\, k \\
                          w_\epsilon & {\rm if}\,\,   \  x\in Y_{2}(k)\  {\rm for\, all }\, k\\
            \end{cases}
\end{equation}
Clearly $\|\Phi-\Phi_\epsilon\|<\epsilon$. However, $\Phi_\epsilon$ has a unique (ergodic) localized measure of maximal entropy at $w_\infty$ and $\Phi$ has precisely two ergodic localized measures of maximal entropy at $w_\infty$.\\
{\rm (ii) } In a forthcoming note we construct examples of H\"older continuous potentials  $\varphi, \Phi$ with countable infinitely many ergodic localized equilibrium states  at some boundary point  of $\R(\varphi,\Phi)$. We refer to \cite{KW2} for details.
\end{remarks}

%The following example shows that the uniqueness of a localized equilibrium state at the boundary is in general not preserved under small changes of the potential $\Phi$.

Next, we consider interior localized equilibrium states. Recall our standing assumptions that $\Phi=(\phi_1,\cdots,\phi_m)\in C^\delta(X,\bR^m)$ for some $\delta>0$ and that $\dim \R(\Phi)=m$ (i.e. no non-trivial linear combination $t\cdot \Phi=(t_1,\cdots,t_m)\cdot \Phi$ is cohomologous to a constant). Let $\varphi\in C^\delta(X,\bR)$. We first consider the case $\dim\R(\varphi,\Phi)=m$. As noted before, this means that either
$\varphi$ is cohomologous to a constant or $\varphi$ is cohomologous to some non-trivial linear combination $t\cdot \Phi$. It follows from the convexity of $\R(\varphi,\Phi)$ that $I_w\eqdef\{\int \varphi\ d\mu: \mu\in \cM_\Phi(w)\}$ is a singleton. In particular, $\mu\in\cM_\Phi(w)$ is a localized equilibrium state of $\varphi$ with respect to $\Phi$ and $w$ if and only if $\mu$ is a localized measure of maximal entropy at $w$. For $t=(t_1,\cdots,t_m)$ let us denote by $\mu_t$ the (classical) equilibrium state of the potential $t\cdot \Phi$ (which is well-defined by property 4 of (STP)). In \cite{KW} we proved the following result.

\begin{theorem}

Let $w\in {\rm int}\ \R(\Phi)$ and assume $I_w$ is a singleton. Then there exits a unique localized measure of maximal entropy $\mu$ at $w$. Moreover, $\mu=\mu_t$ for some uniquely defined $t\in\bR^m$.

\end{theorem}

We now consider the case $\dim\R(\varphi,\Phi)=m+1$. Let $w\in {\rm int}\ \R(\Phi)$. It follows from the compactness and the convexity of $\R(\varphi,\Phi)$ that
there exist $\bmin < \bmax$ such that

\begin{equation}\label{p00}
I_w=\left\{\int \varphi \ d\mu: \mu\in \cM_\Phi(w)\right\}=\R(\varphi,\Phi)\cap \bR\times \{w\}=[\bmin,\bmax].
\end{equation}

For $(s,t)=(s,t_1,\cdots,t_m)\in \bR^{m+1}$ let $\mu_{s,t}$ denote the uniquely defined (classical) equilibrium measure of the potential $s\varphi+t\cdot \Phi$. In \cite{KW} we showed that the map $F:\bR\times\bR^m\to {\rm int}\ \R(\varphi,\Phi)$ defined by

\begin{equation}\label{p11}
F(s,t)=\left(\int \varphi \ d \mu_{s,t}, \int \phi_1\ d \mu_{s,t},\cdots, \int \phi_m\ d \mu_{s,t}\right)
\end{equation}
is a real-analytic diffeomorphism and that $\mu_{s,t}$ is the unique measure satisfying

\begin{equation}\label{p22}
h(s,t)\eqdef h_{\mu_{s,t}}(f)=\sup\{h_\nu(f): \rv(\nu)=F(s,t)\}.
\end{equation}

\noindent
Moreover, the map $(s,t)\mapsto h(s,t)$ is
 real-analytic. For $\alpha\in (\bmin,\bmax)$ we write $g(\alpha)=g_w(\alpha)=F^{-1}(\cdot,w)(\alpha)$.

\begin{proposition}\label{protriv}

Let $w\in {\rm int}\ \R(\Phi)$. Then the map $\alpha\mapsto g(\alpha)$ is a real-analytic diffeomorphism onto its image and $\mu_{g(\alpha)}$ is the unique measure satisfying\begin{equation}
h_{\mu_{g(\alpha)}} + \int \varphi\ d \mu_{g(\alpha)}=\sup\left\{h_\mu(f)+ \int \varphi \ d\mu: \mu\in \cM_\Phi(w), \int \varphi\ d\mu= \alpha\right\}.
\end{equation}
In particular, if $\mu$ is an  interior ergodic localized equilibrium state of $\varphi$ with respect to $\Phi$ and $w$, then there exists a unique $\alpha\in (\bmin,\bmax)$ with
$\mu=\mu_{g(\alpha)}$.

\end{proposition}

\begin{proof}

The statement is a direct consequence of \eqref{p00},\eqref{p11} and \eqref{p22}.
\end{proof}
Finally, we present our main result about interior localized equilibrium states.

\begin{theorem}

Supose that all localized equilibrium states of $\varphi$ with respect to $\Phi$ and $w$ are interior equilibrium states. Then there exists at least one and  at most finitely many ergodic localized equilibrium states of $\varphi$ with respect to $\Phi$ and $w$. All these ergodic localized equilibrium states  are classical equilibrium states.

\end{theorem}

\begin{proof}

Since there exists a localized equilibrium state of $\varphi$, we may conclude from Proposition \ref{protriv} the existence of an ergodic localized equilibrium state of $\varphi$ with respect to $\Phi$ and  $w$.
Suppose there exist infinitely many ergodic localized equilibrium states of $\varphi$. Again by Proposition
\ref{protriv} there exists a pairwise disjoint sequence $(\alpha_k)_{k\in\bN}\subset (\bmin,\bmax)$ such that each $\mu_{g(\alpha_k)}$ is an
ergodic localized equilibrium state of $\varphi$. Let $\mu$ be a weak$^\ast$ accumulation point of the measures $\mu_{g(\alpha_k)}$. It follows that $\mu$ is also a localized equilibrium state of $\varphi$.
Recall that there are no localized equilibrium states at the boundary Thus, Proposition \ref{protriv} implies that $\mu=\mu_{g(\alpha)}$ for some $\alpha\in (\bmin,\bmax)$. We conclude that the function $\alpha\mapsto h_{\mu_{g(\alpha)}} + \int \varphi\ d \mu_{g(\alpha)}$ is constant on a non-discrete subset of $(\bmin,\bmax)$. Hence, $\alpha\mapsto h_{\mu_{g(\alpha)}} + \int \varphi\ d \mu_{g(\alpha)}$ is constant by the identity theorem. Thus, $\mu_{g(\alpha)}$ is a localized equilibrium state of $\varphi$ with respect to $\Phi$ and $w$ for every $\alpha\in (\bmin,\bmax)$. But this implies that there must exist a localized equilibrium state of $\varphi$ with respect to $\Phi$ and $w$ at the boundary which is a contradiction.

\end{proof}


\begin{thebibliography}{99}

\bibitem{BG}
L. Barreira and K. Gelfert, \emph{Dimension estimates in smooth dynamics: a survey of recent results}, Ergodic Theory and Dynamical Systems {\bf 31} (2011), 641–-671.

\bibitem{BSS}L. Barreira, B. Saussol and J. Schmeling \emph{Higher-dimensional multifractal analysis}, J. Math. Pures Appl. {\bf 9} (2002), 67–-91.

\bibitem{BS} L. Barreira and B. Saussol,  \emph{Variational principles and mixed multifractal spectra}, Trans. Amer. Math. Soc. \textbf{353} (2001), 3919–-3944.


%\bibitem{B}T. Bousch, \emph{Le poisson n'a pas d'aretes}, Annales de l'Institut Henri Poincar\'e (probabilit\'es et statistiques) \textbf{36} (2000),  489--508.

\bibitem{Bo1}R. Bowen,  Equilibrium states and the ergodic
theory of Anosov diffeomorphisms, Lecture Notes in Math. {\bf
470}, Springer-Verlag, Berlin, 1975.

\bibitem{B1}R. Bowen,  \emph{Hausdorff dimension of quasicircles},  Inst. Hautes \'Etudes Sci. Publ. Math. {\bf 50} (1979), 11--25.

\bibitem{Bo2}R. Bowen, \emph{Some systems with unique equilibrium states}, Math. Systems Theory \textbf{8} (1974/75), 193–-202.

\bibitem{C}V. Climenhaga \emph{Topological pressure of simultaneous level sets}, Nonlinearity \textbf{26}, (2013), 241–-268.

\bibitem{CT}V. Climenhaga and Thompson, \emph{Equilibrium states beyond specification and the Bowen property}, Journal London Mathematical Society (2) \textbf{87} (2013), 401–-427.

\bibitem{DGS}M. Denker,  C. Grillenberger and- K. Sigmund,
\emph{Ergodic theory on compact spaces},
Lecture Notes in Mathematics, Vol. 527. Springer-Verlag, Berlin-New York, 1976. iv+360 pp.

%\bibitem{GW1}K. Gelfert and C. Wolf,
%\emph{Topological pressure via periodic points}, Trans. Amer. Math. Soc. \textbf{360} (2008),  545-561.
%\bibitem{GW2}K. Gelfert and C. Wolf, \emph{On the distribution of periodic orbits}, Discrete Contin. Dyn. Syst. \textbf{26} (2010), 949–966.
%\bibitem{GM}W. Geller and M. Misiurewicz, \emph{Rotation and entropy}, Trans. Amer. Math. Soc., \textbf{351} (1999), 2927-2948.
%\bibitem{grisvard}P. Grisvard, \emph{Elliptic problems in nonsmooth domains, Monographs and Studies in Mathematics} vol. 24,
%Pitman (Advanced Publishing Program), Boston, MA, 1985

\bibitem{HRu}N. Haydn and D. Ruelle,
\emph{Equivalence of Gibbs and equilibrium states for homeomorphisms satisfying expansiveness and specification},
Comm. Math. Phys. \textbf{148} (1992), 155–-167.

\bibitem{F}H. Federer, \emph{Geometric measure theory}, Springer-Verlag Berlin Heidelberg New York, 1996.

\bibitem{FDPV}T. Fisher, L. Diaz, M. Pacifico and J. Vieitez, \emph{Entropy-expansiveness for partially hyperbolic diffeomorphisms},  Discrete Contin. Dyn. Syst. \textbf{32} (2012), 4195–-4207.

\bibitem{He} L. He, J. Lv and L. Zhou, \emph{Definition of measure-theoretic pressure using spanning sets}, Acta Math. Sinica, English series \textbf{20} (2004), 709--718.

\bibitem{Je} O.~Jenkinson, \emph{Rotation, entropy, and equilibrium states}, Trans. Amer. Math. Soc. \textbf{353} (2001), 3713--3739.
%\bibitem{Je1} O.~Jenkinson, \emph{Geometric Barycentres of Invariant Measures for Circle Maps}, Ergodic Theory and Dynamical Systems \textbf{21} (2001), 511--532.
%\bibitem{Je2} O. Jenkinson, \emph{Frequency locking on the Boundary of the Barycentre Set}, Experimental Mathematics \textbf{9} (2000), 309--317.
\bibitem{Kat:80} A.~Katok, \emph{Lyapunov exponents, entropy and periodic
    orbits for diffeomorphisms}, Publ. Math., Inst. Hautes
  \'Etud. Sci. \textbf{51} (1980), 137--173.

\bibitem{KH}A. Katok and B. Hasselblatt, \emph{Introduction to the modern theory of dynamical systems. With a supplementary chapter by Katok and Leonardo Mendoza}, Encyclopedia of Mathematics and its Applications, {\bf 54}, Cambridge University Press, Cambridge, 1995.

\bibitem{Ke}G. Keller, \emph{Equilibrium states in ergodic theory}, London Mathematical Society Student Texts, \textbf{42}, Cambridge University Press, Cambridge, 1998. x+178 pp.

%\bibitem{KMG}S. Kim, R. S. MacKay and J. Guckenheimer, \emph{Resonance regions for families of torus maps}, Nonlinearity 2, (1989), 391-404.

\bibitem{KW}T. Kucherenko and C. Wolf, \emph{Geometry and entropy of generalized rotation sets}, Israel Journal of Mathematics, to appear.

\bibitem{KW2}T. Kucherenko and C. Wolf, \emph{On finitelness of localized equilibrium states}, in preparation.

%\bibitem{Kw1}J. Kwapisz, \emph{Every convex polygon with rational vertices is a rotation set}, Ergodic Theory and Dynamical Systems \textbf{12} (1992), 333--339.
%\bibitem{Kw2}J. Kwapisz, \emph{A toral diffeomorphism with a nonpolygonal ratation set},  Nonlinearity 8, (1995), 461-476.

%\bibitem{Kw3}J. Kwapisz, \emph{A priori degeneracy of one-dimensional rotation sets for periodic point free torus maps}, Trans. Amer. Math. Soc. 354 (2002), 2865-2895

\bibitem{Lo} P. Loeb, \emph{On the Besicovitch covering theorem}, SUT J. Math. \textbf{25} (1989),  51-–55.

\bibitem{MM} A.~Manning and H.~McCluskey, \emph{Hausdorff dimension
for horseshoes}, Ergodic Theory Dynamical Systems \textbf{3} (1983),
251--260.


%\bibitem{Ma}P. Mattila, \emph{Geometry of sets and measures in Euclidean spaces. Fractals and rectifiability}, Cambridge University Press, Cambridge, U.K., 1995.

\bibitem{Mi1}M. Misiurewicz, \emph{Rotation Theory}, Misiurewicz's webpage.


%\bibitem{Mi2}M. Misiurewicz, \emph{Diffeomorphism without any measure with maximal entropy}, Bull.
%Acad. Polon. Sci., Ser. sci. math., astr. et phys, \textbf{21} (1973), 903-910.

%\bibitem{MZ}M. Misiurewicz and K. Ziemian, \emph{Rotation sets and ergodic measures for torus homeomorphisms}, Fundam. Math. \textbf{137} (1991), 45--52.


\bibitem{N} S.~Newhouse, \emph{Continuity properties of entropy},  Ann. of
  Math. (2)  \textbf{129} (1989), 215--235.

\bibitem{P} Y. Pesin, \emph{Dimension Theory in Dynamical Systems: Contemporary Views and Applications}, Chicago Lectures in Mathematics, Chicago University Press, Chicago, 1997.
    
\bibitem{PP}Ya. Pesin and B. Pitskelʹ, \emph{Topological pressure and the variational principle for noncompact sets}, (Russian) Funktsional. Anal. i Prilozhen. \textbf{18} (1984), 50–-63.

\bibitem{Po}H. Poincar\'e, \emph{Sur les cousbes d\'efinies par les \'equations diff\'erentielles}, Euvres Compl\`etes, tome 1, Gauthier-Villars, Paris, (1952), 137-158.

\bibitem{PU} F.~Przytycki and M.~Urbanski, \emph{Conformal fractals: ergodic theory methods},
London Mathematical Society Lecture Note Series, 371. Cambridge University Press, Cambridge, 2010. x+354 pp.

\bibitem{Ru2}D. Ruelle, \emph{Repellers for Real Analytic Maps},
Ergodic Theory and Dynamical Systems {\bf 2} (1982), 99--107.

\bibitem{Ru} D.~Ruelle, \emph{Thermodynamic Formalism}, Cambridge:
  Cambridge Univ. Press, 2004.
  
\bibitem{TV} F. Takens and E. Verbitskiy, \emph{On the variational principle for the topological entropy of certain non-compact sets}, Ergodic Theory Dynam. Systems \textbf{23} (2003), 317–-348.
  
\bibitem{T}D. Thompson, \emph{A thermodynamic definition of topological pressure for non-compact sets}, Ergodic Theory and Dynamical Systems \textbf{31} (2011), 527--547.

\bibitem{Wal:81} P.~Walters, \emph{An introduction to ergodic theory},
  Graduate Texts in Mathematics 79, Springer, 1981.
\bibitem{Z}K. Ziemian, \emph{Rotation sets for subshifts of finite type}, Fund. Math. \textbf{146} (1995), 189--201.
\end{thebibliography}
\end{document}